\documentclass[a4paper,11pt]{article}
\RequirePackage{amsthm,amsmath,amsfonts}

\newtheorem{theorem}{Theorem}[section]

\newtheorem{proposition}{Proposition}[section]

\newtheorem{remark}{Remark}%[section]
\newtheorem{assumption}{Assumption}

\newcommand{\mI}{{\mathcal I}}
\newcommand{\mJ}{{\mathcal J}}
\newcommand{\mK}{{\mathcal K}}
\newcommand{\mM}{{\mathcal M}}

\newcommand{\bN}{{\mathbb N}}
\newcommand{\bR}{{\mathbb R}}
\newcommand{\bZ}{{\mathbb Z}}
\newcommand{\bP}{{\mathbb P}}
\newcommand{\bE}{{\mathbb E}}
\newcommand{\bI}{{\mathbb I}}

\newcommand{\bm}{\bar{m}}
\newcommand{\bM}{\bar{M}}

\newcommand{\argmax}{\operatornamewithlimits{argmax}} 

\title{Utility Optimization in Congested Queueing Networks}
\author{N. S. Walton}

\begin{document}
\maketitle

\begin{abstract}
We consider a multi-class single server queueing network as a model of a packet switching network. The rates packets are sent into this network are controlled by queues which act as congestion windows. By considering a sequence of congestion controls, we analyse a sequence of stationary queueing networks. In this asymptotic regime, the service capacity of the network remains constant and the sequence of congestion controllers act to exploit the network's capcity by increasing the number of packets within the network. We show the stationary throughput of routes on this sequence of networks converges to an allocation that maximizes aggregate utility subject to the network's capacity constraints. To perform this analysis, we require that our utility functions satisfy an exponential concavity condition. This family of utilities includes weighted $\alpha$-fair utilities for $\alpha >1$.
\end{abstract}

\section{Introduction}
% we consider the behaviour of queueing networks under a limit where these networks become congested.
We are interested in proving how end-to-end control can provide a mechanism where routes receive a transfer rate that is the solution to a network wide utility optimization problem. Using differential equations to model network dynamics, authors have demonstrated network utility optimization, see Kelly, Maulloo, Tan \cite{KMT98}. In this paper, we consider a sequence of stationary queueing networks and demonstrate network utility optimization. We note here that an accessible heuristic derivation of the papers main results can be found in Section \ref{heuristic}.

%Initial work on effective bandwidths \cite{GH91, Ke91i} considered the demand different traffic sources $i\in\mI$ imposed on a queue $j\in\mJ$ of capacity $C_j$. A large deviations analysis showed the bandwidth required by traffic sources imposed a constraint
%\begin{equation*}
% \sum_{i\in\mI} \Lambda_i \leq C_j,
%\end{equation*}
%where $\Lambda_i$ is expressible in terms of the log moment generating function of the traffic load of source $i\in\mI$. 

As a method of allocating resources and introducing fairness, work of Kelly \cite{Ke97} considered utility optimization of the form
\begin{align}
 &\text{maximize} && \sum_{i\in\mI} U_i(\Lambda_i)&\label{system1}\\
&\text{subject to} && \sum_{i: j\in i} \Lambda_i \leq C_j, &j\in\mJ,\label{system2}\\
&\text{over} && \Lambda_i\geq 0, & i\in\mI\label{system3},
\end{align}
where $U_i$ is an increasing strictly concave utility function with derivative satisfying $U'_i(\Lambda_i)\rightarrow \infty$ as $\Lambda_i\rightarrow 0$. We call this optimization problem the \textit{system problem}. In words, it states that one should maximize the aggregate utility of the transfer rate received by users of a network's routes \eqref{system1} subject to the network's capacity constraints \eqref{system2}. But, the utility preferences of users are separate and unknown to the functional operation of a communication network. Similarly, users do not explicitly know the network's topology and exact capacity. So, the network's behaviour and preferences of users must be decomposed.

In \cite{Ke97}, Kelly also introduced \textit{proportional fairness} as the unique solution to the optimization problem
\begin{align}
 &\text{maximize} && \sum_{i\in\mI} \bar{m}_i\log \Lambda_i&\label{network problem1}\\
&\text{subject to} && \sum_{i: j\in i} \Lambda_i \leq C_j, &j\in\mJ,\\
&\text{over} && \Lambda_i\geq 0, & i\in\mI.\label{network problem3}
\end{align}
We call this optimization problem the \textit{network problem} or the \textit{proportionally fair optimization problem}. Kelly \cite{Ke97} considered the combined solution of the network problem and the following \textit{user problems}, for each $i\in\mI$
\begin{align}
&\text{maximize}&& U_i\big(\frac{\bar{m}_i}{q_i}\big)-\bar{m}_i&\label{user problem}\\
&\text{over}&& \bar{m}_i\geq 0.\notag&
\end{align}
This combined solution was considered under the relation
\begin{equation}
 \bm_i=\Lambda_iq_i,\qquad i\in\mI, \label{little}
\end{equation}
where $q_i=\sum_{j\in i} q_j$ and $(q_j:j\in\mJ)$ are the Lagrange multipliers associated with the network problem. Theorem 2 of Kelly \cite{Ke97} found under (\ref{little}) that the combined solution of the network and user problems gave the solution to the system problem. 

This result was constructed to suggest an end-to-end argument for providing optimization and fairness across a communication network. The result provided a method for decomposing the system problem into a user problem that is independent of the network structure except through parameter $q_i$ and a network problem that is independent of user's preferences, except through parameters $\bm=(\bm_i:\,i\in\mI)$. Interpreted in the context of a communication network this separated the preferences of users performing end-to-end communication and the network's preferred optimal behaviour. In \cite{Ke97} the solution is interpreted as setting prices $(q_j:j\in\mJ)$ for sending traffic through the network. With these prices each user, $i\in\mI$, chooses an amount of money, $\bm_i$, it is willing to pay per unit of time for network resources. From this, the user receives an amount of bandwidth $\Lambda_i=\frac{\bm_i}{q_i}$.

By construction, this result considers a static model and the end-to-end argument performed by users is implicit. Subsequent work has successfully used differential equations to add dynamics to this notion of optimization and decomposition \cite{KMT98, JoTa01, Ke03, Sr04}. Other work has considered the form of utility optimization achieved by different protocols \cite{VBB00, MoWa00, KuSr03}. Also, authors have also considered stochastic models of flow across a network \cite{MaRo98, BoMa01, BoPr04}. More recently, authors have explicitly used queue length as a mechanism to provide utility based fairness \cite{St05,ErSr07}. We now describe the approach taken in this paper.

In 1979, Schweitzer \cite{Sc79} studied approximations of closed multi-class queueing networks and considered how asymptotic conditions on such networks might satisfy the Kuhn-Tucker conditions for proportionally fair optimization. In 1989, Kelly \cite{Ke89} studied approximations of closed queueing networks and by an analogous analysis considered a similar optimisation formulation. In 1999, Massouli\'e and Roberts \cite{MaRo99} studied a fluid type queueing model and used these same Kuhn-Tucker conditions to deduce proportional fairness. Using large deviations and also heavy traffic analysis, recent work of Walton \cite{Wa09} and Kelly, Massouli\'e and Walton \cite{KMW09} have provided rigorous formalisations of the relationship between closed queueing networks and proportional fairness. The large deviations connection between multi-class queueing networks and proportional fairness gives a much more literal meaning to the network problem (\ref{network problem1}-\ref{network problem3}).

In light of this work, the first key observation of this paper is that we can use the asymptotic behaviour of a multi-class queueing network to express the network problem (\ref{network problem1}-\ref{network problem3}). Given this, $(\bar{m}_i:i\in\mI)$ must represent the number of packets on each route of this network. We let $(\bar{m}_i:i\in\mI)$ be recorded and controlled by congestion windows. In this paper, for each route a congestion window sends packets along its route at a rate which is a function of the number of packets on that route. We call this system of queues and congestion windows, a \textit{queueing system}. 

We wish to associate the congestion windows in a queueing system with the user problem (\ref{user problem}). The second key observations of this paper is that the user problem (\ref{user problem}) is reminiscent of a Legendre-Fenchel transform. Results like the G\"artner-Ellis Theorem \cite{DeZe98} relate the large deviations behaviour of sequences of random variables to the Legendre-Fenchel transform of their log moment generating function. So by controlling the number of packets in transfer in a network, under a large deviations asymptotic, this Legendre-Fenchel transform observation can be used to associate a utility function with a sequence of congestion windows. We interpret each congestion window in this sequence as a congestion controller's response to the level of congestion within the queueing system. We discuss this point further in Section \ref{heuristic} and more formally in Section \ref{cwnd}.

A third observation is that in our queueing system the statement (\ref{little}) corresponds to the statement of Little's law, i.e. the expected number of packets across routes equals the expected sojourn time of packets through the multi-class queueing network multiplied by the rate packets are sent into the multi-class queueing network. Thus these three observations now place the work of \cite{Ke97} in the context of a network of queues with congestion windowing.

The above three observations prescribe the limiting regime and theoretical approach of this paper. We consider a sequence of stationary queueing systems. The capacity of queues within this network are constant, but the control policy used by each congestion controller is altered in this sequence. Controller's sequentially increase the number of packets in flight within the network. Large deviations results are then be employed to show that the corresponding sequence of stationary distributions asymptotically concentrate probability on a point that maximizes network utility.

We now locate the main results of this paper. The queues and congestion windows considered in our queueing system are quasi-reversible. Theorem \ref{quasi system} describes standard reversibility results can be applied to calculate these queueing systems' stationary distribution. We consider a sequence of stationary queueing systems associated with a particular sequence of congestion windows. We study the large deviations of the stationary distribution of this sequence of queueing systems. In this large deviations limit our above three observations are realized and thus these queueing systems are asymptotically able to execute the analysis of Kelly \cite[Theorem 2]{Ke97}.  We find in our analysis that we require each utility function to satisfy an \textit{exponential concavity condition}, that the map $\lambda\mapsto U_i(e^\lambda)$ is concave. Theorem \ref{primaldual theorem} states this sequence of queueing systems obey a large deviations principle with rate function given by the system problem.
\begin{equation*}
 \max_{\Lambda\in\bR_+^I}\quad \sum_{i\in\mI} U_i(\Lambda_i) \quad\text{ subject to }\quad\sum_{i:j\in i} \Lambda_i \leq C_j,\quad j\in\mJ.
\end{equation*}
From this in Theorem \ref{final thrm}, we prove that the stationary rate packets are transferred through these queueing systems converges to $(\Lambda^*_i:i\in\mI)$ the solution to the system problem (\ref{system1}-\ref{system3}).

%We note that now, instead of interpreting $q$ and $\bm$ as prices and wealth per unit time, we interpret them as round-trip times and congestion window sizes. In addition, the large deviations analysis taken in this paper leads us to consider analogues with the theory of effect bandwidths.

%In this proposed queueing system congestion windows record the number of sent but not yet acknowledged packets on each route of a multi-class queueing network and send packets into this network at a rate which is a function of this number. The queues and congestion windows considered in this queueing system are quasi-reversible. Thus, as Theorem \ref{quasi system} describes, standard reversibility results can be applied to calculate the queueing systems stationary distribution. We allow a sequence of congestion windows to congest the multi-class queueing network and we study the large deviations behaviour of the stationary distribution of the queueing system. Noting the user problem (\ref{user problem}) is reminiscent of a Legendre-Fenchel transform we allow a sequence of congestion windows to solve a modified user problem (\ref{Gi}). As discussed above when congested the multi-class queueing network will solve the network problem. The relation (\ref{little}) will be satisfied as it describes Little's Law for the number of packets on transfer on each route. Thus the queueing system will asymptotically satisfy the network problem, the user problem and relation (\ref{little}). Therefore given Theorem 2 of Kelly, we expect our queueing system to optimize the system problem.

\subsection{Organization}
In Section \ref{heuristic}, we present a heuristic derivation of the papers main results. This section should be quick and accessible to most readers.

In Section \ref{cwnd}, we define our model of a congestion window. We study its stationary distribution when operating in isolation of a network. We study the large deviations behaviour of a sequence of congestion windows and we consider how to associate a utility function with this sequence. 

In Section \ref{multi-class queue}, we consider a well known model of a multi-class queue. We study these queues' stationary distribution when operating in isolation of a network and we study their large deviations behaviour. 

In Section \ref{queueing system}, we connect together the congestion windows of Section \ref{cwnd} and the queues of Section \ref{multi-class queue} to form a \textit{queueing system}. Similarly we study the stationary distribution of this queueing system and we study the large deviations of a sequence of queueing systems. In addition, we study dual relationship between the state and the flow through the queueing system and we prove the stationary throughput of packets in the queueing system converges to the solution to the system problem.

\subsection{Notation}
We let finite set $\mJ$ index the set of queues in a network. Let $J=|\mJ|$. A route through the network is a non-empty subset of queues. Let $\mI$ be the set of routes and let $I=\mI$. For each route $i=\{j^i_1,...,j^i_{k_i} \}$ we associate an order $(j^i_1,...,j^i_{k_i})$. This is the order that route $i$ packets with traverse their route. Also, we define the set of queue-route incidences, $\mK=\{(j,i)\in\mJ \times \mI: j\in i\}$ and let $K=|\mK|$.

Our multi-class queueing system will process packets through a network of queues and congestion windows. For each route there is a congestion window. Let \textit{window size} $\bm_i$ be the number of packets in congestion window $i\in\mI$. The window size is the number of sent but not yet acknowledged packets on route $i$. Each queue $j$ processes packets form routes $i\in\mI$ with $j\in i$. Let $m_{ji}$ be the number of route $i$ packets at queue $j$. We also define,
\begin{equation}
 m_j=\sum_{i:j\in i} m_{ji},\qquad j\in\mJ,\label{mj}
\end{equation}
as the number of packets at queue $j$. As each congestion window records the number of packets in transfer in the queueing network,
\begin{equation}
 \bm_i=\sum_{j\in i} m_{ji},\qquad i\in\mI.\label{mbar}
\end{equation}
Although we will often use $\bm=(\bm_i:i\in\mI)\in\bZ_+^I$, $m=(m_{ji}:(j,i)\in\mK)\in \bZ_+^K$ to refer to the number of packets at congestion windows and queues, in Sections (\ref{LDP cwnd}), (\ref{LDP queue}) and (\ref{LDP system}), when applying large deviations results, we will use $\bm\in\bR_+^I$ and $m\in\bR_+^K$ to refer to the proportion of packets in the queueing system at different congestion windows and queues.

In addition for $m\in\bZ_+^K$ we define,
\begin{equation*}
 \left( \begin{array}{cc} m_j \\ m_{ji}: i\ni j \end{array} \right)=\frac{m_j!}{\prod_{i: j\in i} (m_{ji}!)}.
\end{equation*}
For vectors $x\in\bR_+^D$ and $\phi\in\bR_+^D$ we define $||x||=(\sum_{d=1}^D x_d^2 )^{1/2}$, the Euclidean norm of $x$; we define $\lfloor x \rfloor=(\lfloor x_d \rfloor: d=1,...,D)$, the low integer part of each component of $x$ and we define $\phi\cdot x= \sum_{d=1}^D \phi_d x_d$, the dot product of $\phi$ and $x$.

\section{A Heuristic Derivation of Results}\label{heuristic}
In this section, we give a heuristic derivation of the results in this paper. The heuristic is an adaption of the arguments of Schweitzer \cite{Sc79}, Kelly \cite{Ke89}, and Roberts and Massouli\'e \cite{MaRo99} applied to a modified version of Kelly \cite{Ke97}. The formal proof, given in subsequent sections, follows a similar approach to Walton \cite{Wa09}.

The KKT conditions for the system problem (\ref{system1}-\ref{system3}) are that there exists positive multipliers $(q_j:j\in\mJ)$ and positive rates $(\Lambda_i:i\in\mI)$ such that
\begin{align*}
 U'_i(\Lambda_i)=\sum_{j\in i} q_j,\quad i\in\mI,\\
q_j\bigg(C_j-\sum_{i: j\in i} \Lambda_i\bigg)=0,\quad j\in\mJ,\\
\sum_{i: j\in i} \Lambda_i \leq C_j,\quad j\in\mJ.
\end{align*}

Consider a queueing system with routes $\mI$ and queues $\mJ$, as notated in the previous section. For each route $i$, packets are injected into the network; traverse the queues in their route in order $j^i_1,...,j^i_{k_i}$, and then leave. Let $q_j$ be the stationary sojourn time of a packet at queue $j$; let ${m}_{ji}$ be the stationary number of route $i$ packets in transfer at queue $j$ and let $\Lambda_i$ be the stationary sending rate of route $i$ packets into the network. By Little's law
\begin{equation*}
\Lambda_iq_j={m}_{ji}, \quad \forall i\in \mI\text{ and } j\in i.
\end{equation*}
Summing over the queues on route $i$ and rearranging gives that
\begin{equation}\label{KT-1}
\frac{\bar{m}_i}{\Lambda_i}=\sum_{j\in i} q_j,\quad \forall i\in\mI,
\end{equation}
where $\bar{m}_i$ is the stationary number of packets in transfer on route $i$. Suppose that a congestion window for each route $i$ injects packets into the network at a rate that is a function of $\bar{m}_i$, i.e. $g_i(\bar{m_i})=\Lambda_i$. If we chose $g_i(\cdot)$ so that
\begin{equation}\label{KT0}
\bar{m}_i=g^{-1}_i(\Lambda_i)=\Lambda_iU_i(\Lambda_i)
\end{equation}
then, \eqref{KT-1} implies
\begin{equation}\label{KT1}
 U'_i(\Lambda_i)= \sum_{j\in i} q_j,\qquad i\in\mI.
\end{equation}
Assuming all queues are stable, we know
\begin{equation}\label{KT2}
\sum_{i: j\in i} \Lambda_i \leq C_j,\quad \forall j\in\mJ.
\end{equation}
One can imagine, if equality (\ref{KT2}) is strict then sojourn times will be small, $q_j\approx 0$. Thus, approximately,
\begin{equation}\label{KT3}
q_j\Big( C_j - \sum_{i: j\in i} \Lambda_i \Big) = 0 ,\quad \forall j\in\mJ.
\end{equation}
Also
\begin{equation}\label{KT4}
q_j\geq 0,\quad \forall j\in\mJ \quad\text{and}\quad \Lambda_i\geq 0, \quad \forall i\in\mI.
\end{equation}
Interpreting $(q_j:j\in\mJ)$ as Lagrange multipliers, (\ref{KT1})-(\ref{KT4}) are precisely the above KKT conditions for the system problem. So the rates $(\Lambda_i:i\in\mI)$ and sojourn times $(q_j:j\in\mJ)$, implicitly defined by the queueing network, solve the system problem (\ref{system1}-\ref{system3}).

The remainder of this paper is concerned with making the above arguments rigorous. To make conditions like \eqref{KT3} strict, we will require the network to have a high level of congestion. We achieve this by considering a sequence of congestion controls $g_i^{(c)}$, whilst keeping $C_j$, each queues service capacity, fixed. As $c$ increases, the controllers increase the number of packets in transfer and thus more aggressively exploit network capacity. We think of this increase in congestion in an analogous way to the increase rules employed by the Transmission Control Protocol in Internet communications. We will apply large deviations techniques to a sequence of stationary network models, indexed by $c$, and prove that probability concentrates on the system optimal operating point.

\subsection{Choice of $g_i(\bar{m}_i)$ and $G_i(\bar{m}_i)$}\label{choose g}
Later, we will apply large deviations results to certain stationary distributions. For this reason, it will be convenient to express $g_i$ through the theory of convex duality. Let $G_i$ be a differentiable function such that $g_i(\bar{m})=e^{G_i'(\bar{m})}$ then equation (\ref{KT0}) gives that
\begin{equation}\label{G transform}
 G'^{-1}_i(\log \Lambda_i)=\Lambda_i U'_i(\Lambda_i)\quad \text{or}\quad G'^{-1}_i(\lambda_i)=e^{\lambda_i}U_i'(e^{\lambda_i}).
\end{equation}
Now if $F^*$ is the Legendre-Fenchel transform of a concave differentiable function $F$, i.e. $F^*(\lambda)=-\max_\lambda\{ F(\lambda)-\lambda m\}$ then, by construction, $F^*$ has a derivative that inverts the derivative of $F$, i.e. the inverse of ${F^*}' (\cdot) $ equals $F'(\cdot)$. 
Applying this to \eqref{G transform}, with $G_i(\bar{m})=F^*(\bar{m})$ and $F(\lambda_i)=U_i(e^{\lambda_i})$, we have that 
\begin{align}
 G_i(\bar{m})&=-\max_{\lambda_i}\{ U_i(e^{\lambda_i}) - \bar{m}_i \lambda_i \}\notag\\
&=-\max_{\Lambda_i\geq 0} \{ U_i(\Lambda_i) - \bar{m}_i \log \Lambda_i \}\label{G i}
\end{align}
It is precisely by this function $G_i$ that we are able to replace $\bar{m}_i \log \Lambda_i$, the network problem summand, with the system user's utility $U_i(\Lambda_i)$. 

So that duality can invert this operation, that is $U_i(e^{\lambda})=G_i^*(\lambda)$, we require that $U_i(e^{\lambda})$ is a concave function of $\lambda$.

We note that the optimization used to derive $G_i$ \eqref{G i} is different to the user problem \eqref{user problem} derived by Kelly \cite{Ke97}. This is, inessence, because we chose the rate packets are injected based on the number of packets in transfer. This is in contrast to the user problem of Kelly, which could be interpreted as choosing the number of packets in transfer given the current network delay (or round trip time) of packets. A theoretical advantage to our approach is that queueing networks which inject packets based on the number of packets in transfer are more analytically tractable. In particular, quasi-reversibility results can be applied to explicitly give the stationary distribution of such networks.

In the next two sections, we construct the components of these queueing networks: congestion windows and multi-class queues. In Section \ref{queueing system}, we connect these components together to form a sequence of queueing systems that will execute the above heuristic.

\section{Congestion Windows}\label{cwnd}
Congestion windows keep a record of the number of sent but not yet acknowledged packets in a queueing network. The models of congestion windows considered here are reversible and, thus, lend well to product form results when incorporated into a network \cite{Ke89, As03}. We will later connect these congestion windows to the routes of a network of multi-class single server queues.

When in isolation of a network, we define a congestion window at congestion level $c$ as a continuous time Markov chain $(\bM_i^{(c)}(t):t\in\bR_+)$ on $\bZ_+$ with transition rates,
\begin{equation*}
 q(\bm_i,\bm_i')=
\begin{cases}
 g^{(c)}_i(\bm_i) &\text{if }\bm_i'=\bm_i+1,\\
\Lambda_i &\text{if }\bm_i'=\bm_i-1,\\
0 & \text{otherwise}.
\end{cases}
\end{equation*}
The transition $\bm_i\mapsto \bm_i+1$ are thought of as \textit{injecting a packet} into a network, and for $\bm_i>0$, a transition $\bm_i\mapsto \bm_i-1$ are thought of as \textit{acknowledging a packet} that has been transferred. Thus, we think of $\bar{m}_i$ as recording the number of packets currently in transfer. For the purposes of this paper, it will be convenient to define
\begin{equation*}
 \Lambda_i=e^{\lambda_i}\quad\text{and}\quad g^{(c)}_i(\bm_i)=e^{G^{(c)}_i(\bm_i+1)-G^{(c)}_i(\bm_i)},
\end{equation*}
where $\lambda_i\in\bR$ and $G^{(c)}_i:\bR_+\mapsto [-\infty,\infty)$ is a strictly concave function. Thus a congestion window is defined by a constant $\lambda_i$ and a function $(c,\bar{m}_i)\mapsto G_i^{(c)}(\bar{m}_i)$.

\subsection{Reversibility and stationary behaviour}
We now collect a result about the stationary behaviour of congestion windows.
\begin{proposition}\label{quasi cwnd}
A stationary congestion window is reversible with stationary distribution,
\begin{equation}
 \pi^{(c)}_i(\bm_i)=\pi^{(c)}_i(0)e^{G^{(c)}_i(\bm_i)-\lambda_i\bm_i},\quad \bm_i\in\bZ_+.\label{cwnd ed}
\end{equation}
\end{proposition}
\begin{proof}
The result is immediate from the detail balance equations
\begin{equation*}
 \pi^{(c)}_i(\bm_i)=e^{G^{(c)}_i(\bm_i)-G^{(c)}_i(\bm_i-1)-\lambda_i}\pi^{(c)}_i(\bm_i-1)=...=e^{G^{(c)}_i(\bm_i)-\lambda_i m_i}\pi^{(c)}_i(0).
\end{equation*}
\begin{flushright}
$\square$ 
\end{flushright}
\end{proof}

\subsection{Large deviations}\label{LDP cwnd}
We think of the pair $(c,\bar{m}_i)$ as giving the state of a congestion controller, given there is a level of congestion $c$ and the number of packets in transfer is $\bar{m}_i$. As we saw above, if $c$ is fixed a congestion controller describes a reversible Markov chain. In this section, we study the stationary distribution of a sequence of these Markov chains as $c\rightarrow\infty$. As $c$ increases, the stationary number of packets in transfer will increase. This can be interpretted as the congestion controller becoming increasingly aggressive to exploit the network's capacity. In this section, we use large deviation type arguments to study where probability concentrates under this limit.

%The rate at which packets are sent into a network by a congestion window is determined by the size of the congestion window, and the level of congestion in the network. We have denoted the size of the congestion window by $\bm_i$. We will denote the level of congestion in the network by $c\in\bN$. For a sequence of congestions levels $c\in\bN$, we define a sequence of stationary congestion windows and study the large deviations behaviour of their stationary distribution.

%We think of the pair $(c,\bar{m}_i)$ giving the state of a congestion controller, given there is a level of congestion $c$ and the number of packets in transfer

For each $c\in\bN$, let $\bM^{(c)}_i$ be a stationary congestion window defined by $\lambda_i$ and $G^{(c)}_i(\cdot)$, where $G_i^{(c)}(k)=cG_i\big(\frac{k}{c}+\frac{d_i^{(c)}}{c}\big)$ for function $G_i$ strictly concave, differentiable on $(0,\infty)$ with derivative taking all values in $(-\infty, \infty)$ and with $\{d_i^{(c)}\}_{c\in\bN}$ a bounded sequence with values in $\bR$. We define convex function $G^*_i:\bR\rightarrow \bR$ from $G_i$ with the following Legendre-Fenchel transform
\begin{equation*}
 G^*_i(\lambda_i)=\max_{\bm_i\in\bR_+} \{G_i(\bm_i)-\lambda_i\bm_i\}.
\end{equation*}
We also define
\begin{equation*}
 \bm^*_{\lambda_i}=\argmax_{\bm_i\in\bR_+} \{G_i(\bm_i)-\lambda_i\bm_i\}.
\end{equation*}
In the following, proposition we use large deviation arguments for the purpose of identifying the most probably state of the sequence of random variables $\bM^{(c)}$, $c\in\bN$.
\begin{proposition}\label{LDP cwnd prop}
 a)
\begin{equation*}
 \lim_{c\rightarrow \infty} \frac{1}{c}\log \sum_{k=0}^\infty e^{G^{(c)}_i(k)-\lambda_i k} = G^*_i(\lambda_i),\qquad \lambda_i\in\bR_+.
\end{equation*}
b) For $\bm_i\in(0,\infty)$ and a bounded sequence $\{\bar{\sigma}_i^{(c)}\}_{c\in\bN}$ such that $c\bm_i+\bar{\sigma}_i^{(c)}\in\bZ_+$
\begin{equation*}
 \lim_{c\rightarrow\infty} \frac{1}{c}\log \bP(\bM^{(c)}=c\bm_i+\bar{\sigma}_i^{(c)})=G_i(\bm_i)-\lambda_i-G^*_i(\lambda_i).
\end{equation*}
\end{proposition}
\begin{proof}
 To prove a) we wish to verify the principle of the largest term for this infinite sum \cite[Lemma 2.1]{GOW04}. First we show the upper bound. Let $d=\max_{c}|d^{(c)}_i|$. Since $G_i$ is strictly concave, $\forall \delta>0$ letting
\begin{equation*}
\epsilon=\frac{1}{\delta} \min \{G_i(\bm^*_{\lambda_i})-G_i(\bm^*_{\lambda_i} -\delta) +\delta, G_i(\bm^*_{\lambda_i})-G_i(\bm^*_{\lambda_i}+\delta)-\delta \}>0,
\end{equation*}
we have that for all $\bm\in\bR_+$,
\begin{equation*}
 G_i(\bm_i)-\lambda_i\bm_i\leq G^*_i(\lambda_i)-\epsilon (\bm_i-\bm^*_{\lambda_i}) \bI[\bm_i\geq \bm^*_i+\delta]+  \epsilon(\bm_i-\bm^*_{\lambda_i})\bI[\bm_i\leq \bm^*-\delta].
\end{equation*}
Therefore applying the above inequality and comparing the following with a geometric sum we have,
\begin{align}
 &\sum_{k=0}^\infty e^{G^{(c)}_i(k)-\lambda_i k}=\sum_{k=0}^\infty e^{c[G_i(\frac{k}{c}+\frac{d_i^{(c)}}{c})-\lambda_i \frac{k}{c}]}\notag\\
&\leq e^{cG^*_i(\lambda_i)+d^{(c)}_i}\left[ 2(c\delta+d+1)+2\sum_{k\in\bZ_+} e^{c\epsilon(\frac{k}{c}-\bm^*)}\right]\notag\\
&\leq e^{cG^*_i(\lambda_i)+d^{(c)}_i}\left[ 2(c\delta+d+1)+2\frac{e^{-(c\delta-d-1)}}{1-e^{-\epsilon}}\right].\label{finite sum}
\end{align}
Hence as the first term in the square bracket dominates
\begin{equation*}
 \limsup_{c\rightarrow\infty} \frac{1}{c}\log \sum_{k=0}^\infty e^{G^{(c)}_i(k)-\lambda k} \leq G^*_i(\lambda).
\end{equation*}
By proving the lower bound for a) we can simultaneously verify b). Using the terminology of b), for all $\bm_i\in (0,\infty)$,
\begin{align}
 G_i(\bm_i) -\lambda_i\bm_i &=\lim_{c\rightarrow\infty} \frac{1}{c} \log e^{c G_i(\bm_i+ \frac{\bar{\sigma}^{(c)}_i}{c}+\frac{d^{(c)}_i}{c})-\lambda_ic\bm_i+\bar{\sigma}_i^{(c)}}\label{G lower bound}\\
&\leq \liminf_{c\rightarrow \infty} \frac{1}{c} \log \sum_{k=0}^\infty e^{G^{(c)}_i(k)-\lambda_ik}.\notag
\end{align}
Taking $\bm_i=\bm^*_{\lambda_i}$ and $\bar{\sigma}^{(c)}_i=\lfloor \bm^*_{\lambda_i}\rfloor -\bm^*_{\lambda_i}$ verifies a). Given a) and (\ref{G lower bound}) we have b).
\begin{flushright}
$\square$ 
\end{flushright}
\end{proof}

We now emphasize the following duality between state and flow in congestion windows. Suppose as $c\rightarrow \infty$, the number of packets in transfer is approximately $c\bm^*_i$, for some $\bm_i^*\in(0,\infty)$. Thus the flow out of the congestion window is approximately,
\begin{equation*}
g_i^{(c)}(c\bm^*_i)\approx \exp\Big\{\frac{G_i(\bm^*_i+\frac{1}{c})-G_i(\bm^*_i)}{\frac{1}{c}} \Big\}\xrightarrow[c\rightarrow\infty]{} e^{G'_i(\bm^*_i)}.
\end{equation*}
When stationary the average outward flow of packets from the congestion window equals the average inward flow. Thus, we have that $G'_i(\bm^*_i)=\lambda_i$, or in other words,
\begin{equation*}
 \bm^*_i=\argmax_{\bm_i\in\bR_+}\{G_i(\bm_i)-\lambda_i\bm_i\}.
\end{equation*}
By this duality and the correct choice of $G_i$ (as discussed in Section \ref{choose g}),  we can control the throughput of packets from the congestion window so that it optimizes a utility function.

\subsection{Utility optimization}
The rate packets are acknowledged by a congestion window is $\Lambda_i=e^{\lambda_i}$, thus the utility associated with $\lambda_i$ is $U_i(e^{\lambda_i})$. If we wish to maximize the system problem, we must define $G_i$ through the following user problem,
\begin{equation}
 -G_i(\bar{m}_i)=\max_{\lambda_i\in\bR} \{U_i(e^{\lambda_i})-\bar{m}_i\lambda_i\},\qquad \bar{m}_i\in\bR_+\label{Gi}
\end{equation}
and similarly by the duality of Legendre-Fenchel transforms, we may define $U_i$ from $G_i$ by
\begin{equation}
 U_i(e^{\lambda_i})=\min_{\bar{m}_i\in\bR_+} \{\lambda_i\bar{m}_i-G_i(\bar{m}_i)\}=-G_i^*(\lambda_i),\qquad \lambda_i\in\bR.\label{Ui}
\end{equation}
The function $U_i(e^{\lambda_i})$ must be concave as $G^*_i(\lambda_i)$ convex. Thus in order to optimize a utility function $U_i$, we require the assumption,

\begin{assumption}\label{assump 1}
 The utility function $U_i$ is \textit{exponentially concave}, that is the map $\lambda_i\mapsto U_i(e^{\lambda_i})$ is strictly concave on $\bR$. 
\end{assumption}
\noindent We also collect the differentiability assumptions that we make on $G_i$.
\begin{assumption}\label{assump 2}
 We assume $G_i$ defined by (\ref{Gi}) is strictly concave, (continuously) differentiable on $(0,\infty)$ with derivative taking all values in $\bR$.
\end{assumption}
%Although by associating (\ref{Ui}) it is necessary $U_i$ is exponentially concave, assumption \ref{assump 2} is needed in order to simplify our analysis. Given the G\"artner-Ellis Theorem \cite[page 44]{DeZe98}, it seems ingeneral it is necessary that a congestion control mechanism $M^{(c)}_i$ satisfies,
%\begin{equation*}
% \frac{1}{c} \log \bE e^{\theta_i \frac{M_i^{(c)}}{c}}\xrightarrow[c\rightarrow\infty]{} U_i(e^{\lambda_i})-U_i(e^{\lambda_i-\theta_i}),\qquad \theta_i\in\bR.
%\end{equation*}

\begin{remark}[Weighted $\alpha$-fairness, $\alpha>1$]
The weighted $\alpha$-fair family of utility functions considered by Mo and Walrand \cite{MoWa00}, corresponds to the aggregate utility of users with iso-elastic utility, that is utilities,
\begin{equation*}
 U_i(\Lambda_i)=
\begin{cases} 
\frac{w_i\Lambda_i^{1-\alpha}}{1-\alpha} & \text{if } w_i\in\bR_+,\;\alpha>0,\; \alpha\neq 1,\\
w_i\log \Lambda_i & \text{if } w_i\in\bR_+,\;\; \alpha= 1.
\end{cases}
\end{equation*}
The weighted $\alpha$-fair class has proved popular as it contains proportional fairness ($\alpha=w_i=1$), TCP fairness ($\alpha=2$, $w_i=\frac{1}{T^2_i}$), and converges to maximum throughput ($\alpha\rightarrow 0$, $w_i=1$) and max-min fairness ($\alpha\rightarrow\infty$, $w_i=1$).

One can easily verify that $U_i(\Lambda_i)$ is exponentially concave for $\alpha>1$ and that
\begin{equation*}
 G_i(\bar{m}_i)=\frac{\bar{m}_i}{1-\alpha}\log \frac{\bar{m}_i}{w_i} -\frac{\bar{m}_i}{1-\alpha}=\frac{1}{1-\alpha} \int_0^{\bar{m}_i} \log \frac{x}{w_i} dx.
\end{equation*}
So $U_i$ and $G_i$ satisfy the two assumptions above for $\alpha>1$. Thus our results apply for weighted $\alpha$-fairness for $\alpha>1$. The case $\alpha=1$ is considered in \cite{Wa09}. For weighted $\alpha$-fairness $\alpha>1$, a convenient form for $g_i^{(c)}$ and $G_i^{(c)}$ to take is
\begin{equation*}
 g_i^{(c)}(\bar{m}_i)=\left(\frac{cw_i}{\bar{m}_i}\right)^{\frac{1}{\alpha-1}}\quad \text{and}\quad e^{G_i^{(c)}(\bar{m}_i)}=\frac{(cw_i)^{\frac{\bar{m}_i}{\alpha-1}}}{(\bar{m}_i!)^{\frac{1}{\alpha-1}}}.
\end{equation*}
\end{remark}

\section{Multi-class single server queues}\label{multi-class queue}
We define a multi-class single server queue. These queues are quasi-reversible and are described in Kelly \cite{Ke79}. When connected to form a network these queues will process packets over different routes of a queueing system. Congestion windows will regulate the number of packets present on each route. After defining these multi-class queues in this section, we will define this queueing system in Section \ref{queueing system}.

A queue $j\in\mJ$ is fed packets from classes from the set of routes $\{i\in\mI:j\in i\}$. Packets occupy different positions within the queue and have an exponentially distributed mean 1 service requirement. Given there are $m_j\in\bZ_+$ packets at queue $j$, packets occupy positions $1,...,m_j$. The total service devoted to these packets is $C_j\in(0,\infty)$. This fixed service is then divided amongst the different positions in the queue. A proportion $\gamma_j(l,m_j)$ of this capacity is devoted to the packet at position $l=1,...,m_j$. Upon completing its service a packet at position $l$ will depart the queue and the packets at positions $l+1,...,m_j$ will move to positions $l,...,m_j-1$ respectively. In this section we assume packets of route $i$ arrive at queue $j$ as a Poisson process of rate $\Lambda_i$. Upon arrival a packet will move to position $l=1,...,m_j+1$ with probability $\delta_j(l,m_j+1)$ and packets which were in position $l,...,m_j$ will move to positions $l+1,...,m_j+1$.

Let $s^j=(i_1^j,...,i_{m_j}^j)\in\mI^{m_j}$, for $m_j>0$, give the state of queue $j$. Let $T^i_{\cdot, (j,l)}$ denote the arrival of a class $i$ packet to position $l$ in queue $j$ and let $T^i_{(j,l),\cdot}$ denote the departure of class $i$ the packet in position $l$. Thus the state of this queue forms a continuous time Markov chain with transition rates given by,
\begin{equation*}
 q(s^j,s'^j)=
\begin{cases}
 \Lambda_i \delta_j(l,m_j+1) &\text{for } s'^j=T^i_{\cdot, (j,l)} s^j,\; l=1,...,m_j+1,\\
C_j\gamma_j(l,m_j) & \text{for } s'^j=T^i_{(j,l),\cdot} s^j,\; i^j_l=i,\; l=1,...,m_j,\\
0 & \textit{otherwise}.
\end{cases}
\end{equation*}

\subsection{Quasi-reversibility and stationary behaviour}
These queues are known to be quasi-reversible and their stationary distribution is well understood \cite{As03, Ke79}.

\begin{proposition}\label{quasi queue}
provided the following stability condition holds,
\begin{equation*}
 \sum_{i: j\in i} \Lambda_i < C_j,\qquad j\in\mJ,
\end{equation*}
 a multi-class single server queue $j\in\mJ$ is quasi-reversible and $(M_{ji}: i\in \mI, j\in i)$ the stationary number of route $i$  packets at queue $j$ has distribution,
\begin{equation}\label{queue ed}
 \bP(M_{ji}=m_{ji},\: \forall i \ni j)=\left( \frac{C_j-\sum_{i: j\in i} \Lambda_i}{C_j}\right) \left( \begin{array}{cc} m_j \\ m_{ji}\; : \; i\ni j\end{array} \right) \prod_{i : j\in i} \bigg(\frac{\Lambda_i}{C_j}\bigg)^{m_{ji}},
\end{equation}
for  $m_{ji}\in\bZ_+$ for each $i\in\mI$ such that $j\in i$ and where $m_j=\sum_{i: j\in i} m_{ji}$.
\end{proposition}
\begin{proof}
 Let $M_j$ be the number of packets at queue $j$. Since the queue does not discriminate between different packets' classes, $M_j$ is a reversible Markov chain and thus its departures prior to time $t$ form a Poisson process independent of the Poisson process of arrivals after time $t$. By thinning these Poisson processes with probability $\frac{\Lambda_i}{\sum_{j: r \in j}\Lambda_r}$ we obtain the arrival and departure processes of route $i$ packets and thus deduce that the queue is quasi-reversible. The stationary distribution of $M_j$ is geometric with parameter $\frac{\sum_{r: j\in r} \Lambda_i}{C_j}$ combining this with the same thinning argument and summing over states $s^j$ obtaining $(m_{ji}:i\in\mI, j\in i)$ we obtain (\ref{queue ed}). For more details see Kelly \cite[Theorem 3.1]{Ke79} or Asmussen \cite[\S IV.4]{As03}.
\begin{flushright}
$\square$ 
\end{flushright}
\end{proof}

\subsection{Large deviations}\label{LDP queue}
We study the stationary distribution (\ref{queue ed}) when the number of packets of each class is increased proportionately, but all other queueing parameters are kept constant. We collect some large deviations results on these queues in this regime \cite{Pi79,Wa09}.

\begin{proposition}\label{LDP queue prop}
 For $j\in\mJ$, let $(M_{ji}:i\in\mI, j\in i )$ have distribution (\ref{queue ed}).\\
a) Let $m^j=(m_{ji}: i\in\mI,j\in i)$ and $\sigma^{j,(c)}=(\sigma^{(c)}_{ji}: i\in\mI, j\in i)$, $c\in\bN$ be such that $m_{ji}\in\bR_+$, $\sup_{c\in\bN}||\sigma^{j,(c)}||<\infty$ and $cm_{ji}+\sigma_{ji}^{(c)}\in\bZ_+$ $\forall i$ with $j\in i$, then,
\begin{equation*}
 \lim_{c\rightarrow \infty} \frac{1}{c} \log \bP (M_{ji}=cm_{ji} +\sigma_{ji}^{(c)},\: \forall i\ni j)=-\beta_j(m^j),
\end{equation*}
where,
\begin{equation*}
 \beta_j(m^j)=\sum_{\substack{i:j\in i\\ m_j>0}} m_{ji}\log \frac{m_{ji}C_j}{m_j\Lambda_i}.
\end{equation*}
\noindent b) The function $\beta_j(m^j)$ is continuous, convex and is such that
\begin{equation}\label{betaj}
 \inf_{m^j\geq 0} \beta_j(m^j) =
\begin{cases}
 0 &\text{if } \sum_{i: j\in i} \Lambda_i \leq C_j,\\
-\infty &\text{otherwise}.
\end{cases}
\end{equation}
\end{proposition}
\begin{proof}
 \noindent a) Define $\sigma_j^{(c)}=\sum_{i:\; j\in i} \sigma_{ji}^{(c)}$. By Stirling's formula $$ \lim_{c\rightarrow\infty} \frac{1}{c}\log (cm_{ji}+\sigma_{ji}^{(c)})!= m_{ji}\log m_{ji} -m_{ji}.$$ Thus,
\begin{align*}
&\lim_{c\rightarrow\infty} \frac{1}{c} \log \bP(M_{ji}=cm_{ji}+\sigma^{(c)}_{ji}, i\ni j)\\
=& \lim_{c\rightarrow \infty} \frac{1}{c} \left[ \log (cm_j + \sigma_j^{(c)})!  -\sum_{i: j\in i} \log (cm_{ji}+\sigma_{ji}^{(c)})! + \sum_{i: j\in i} (cm_{ji} +\sigma_{ji}^{(c)})\log \frac{\Lambda_i}{C_j}\right]\\
=& -\lim_{c\rightarrow\infty} \sum_{\substack{i: j\in i\\ m_j>0}} m_{ji}\log \frac{(m_{ji}+\frac{\sigma_{ji}^{(c)}}{c})C_j}{(m_j+\frac{\sigma_j^{(c)}}{c})\Lambda_i} =-\beta_j(m^j).
\end{align*}

\noindent b) Taking $x\log x=0$ for $x=0$, $x\log x$ is continuous on $\bR_+$, thus $\beta_j$ is continuous. We next prove (\ref{betaj}). For two probability distributions $p$ and $q$ with the same support on $\mK$, we define their relative entropy to be $D(p||q)=\sum_s p_s\log \frac{p_s}{q_s}$. In particular one can verify
\begin{equation}
 \min_p D(p||q)=0, \label{entropy}
\end{equation}
and is minimized by $p=q$. Thus taking $p=(\frac{m_{ji}}{m_j}: i\in\mI, j\in i)$ and $q=(\frac{\Lambda_i}{\sum_{r: j\in r} \Lambda_r}: i\in\mI,\: j\in i )$,
\begin{align*}
 \inf_{m^j\geq 0} \beta_j(m^j)&=\inf_{m^j> 0} m_j \left( \sum_{i: j\in i} \frac{m_{ji}}{m_j} \log \frac{m_{ji}\sum_{r:j\in r} \Lambda_r}{m_j\Lambda_i} \right) + m_j \log \frac{C_j}{\sum_{r: j\in r} \Lambda_r}\\
& =\inf_{m^j>0} m_j \log \frac{C_j}{\sum_{r: j\in r} \Lambda_r} =
\begin{cases}
 0 & \text{if } \sum_{r: j\in r} \Lambda_r \leq C_j,\\
-\infty &\text{otherwise.}
\end{cases}
\end{align*}
Finally to show that $\beta_j(\cdot)$ is convex one can verify that 
\begin{align}
 \beta_j(m^j) = \max\;\: \sum_{i:j\in i} m_{ji}\phi_i \quad &\text{subject to}\;\: \sum_{i:j\in i} \Lambda_i e^{\phi_i} \leq C_j \notag\\
&\text{over}\quad \phi_i\in\bR, \quad i\ni j.\label{beta LF}
\end{align}
Thus $\beta_j(\cdot)$ is expressible as a Legendre-Fenchel transform and so is convex.
\begin{flushright}
$\square$ 
\end{flushright}
\end{proof}

\section{A queueing system}\label{queueing system}
%The original work on effective bandwidths considered the load different traffic sources imposed on a multi-class queue \cite{Ke91,GH91}.
%This work followed a large deviations analysis. The statistical characteristics of each traffic source defined the queue's ability to multiplex traffic. Analogous to this work 
%In this section we consider a multi-class queueing network and we consider different traffic sources to correspond to different routes through this network. Via congestion windows routes now control the statistical characteristics of traffic in an end-to-end way. We study the large deviations of this network. We consider the stationary rate packets are sent across the network. 

We now connect together the queues and congestion windows discussed in the last two sections to form a network. The interior of this network consists of a set of multi-class queues with routes defined over these queues. We think of this as a simple model of a packet switching network. The number of packets in transfer on a route are determined by congestion windows, as defined in Section \ref{cwnd}. The queueing system, defined by these queues and congestion windows, models the congestion control of a packet switched communication network with a fixed number of document transfers in progress. %These documents are transferred by packets across the different routes of the network. 
%The rate at which packets are sent is controlled in an end-to-end manner by the congestion windows. 

In this section, using quasi-reversibility results and for $c$ fixed, we calculate the stationary distribution of this queueing system. 
For these stationary distributions, we increase $c$, the congestion level of each congestion window and, using large deviations techniques, we calculate where probability concentrates.
We show that the most likely state is given an entropy optimization and its dual is the system problem. In particular, we prove that $\Lambda^{(c)}$ the stationary rate packets traverse the network converges to the solution to the system problem (\ref{system1}-\ref{system3}).

We consider a network of queues indexed by the set $\mJ$ and congestion windows indexed by the set of routes $\mI$. Each queue $j\in\mJ$ will process packets as described in Section \ref{multi-class queue}, but transitions between queues will be prompted by transitions within the queueing system (rather than by a Poisson process). Similarly, each congestion window will send packets into the network as described in Section \ref{cwnd}, but transitions $\bar{m}_i\mapsto \bar{m}_i-1$ will be prompted by the successful transfer of a packet in the queueing network. 

If with route $i$ we associate route order $(j_1^i,...,j_{k_i}^i)$, a packet injected by congestion window $i$ will prompt an arrival at queue $j_1^i$. Also the departure of a route $i$ packet from queue $j^i_{k}$, $k=1,...,k_i-1$ will prompt an arrival at queue $j^i_{k+1}$ and similarly a departure of a route $i$ packet from the final queue $j^i_{k_i}$ will prompt a transition $\bar{m}_i\rightarrow \bar{m}_i-1$ at congestion window $i$. In this way, packets are sent into the network, transferred along their route and finally acknowledged.

We more explicitly describe the state of our queueing system as follows. As in Section \ref{multi-class queue} we let $s^j=(i^j_1,...,i^j_{m_j})$ record the state of queue $j$ and let $s=(s^j:j\in\mJ)$ record the state of our queueing system. Also we let $\bar{m}_i$ record the state of congestion window $i$ and let $\bar{m}=(\bar{m}_i: i\in \mI)$ record the state of our congestion windows. Finally let $\underline{s}=(s,\bar{m})$ record the explicit state of our queueing system. We define the transitions in our queueing network as follows. Let $\underline{s}\mapsto T^i_{\cdot, (j,l)} \underline{s}$ define the transition corresponding to a route $i$ packet injected by congestion window $i$ and arriving at position $l$ in queue $j$. Let $\underline{s}\mapsto T^i_{(j,l), (j',l')} \underline{s}$ denote the departure of a route $i$ packet from position $l$ of queue $j$ which arrives at position $l'$ in queue $j'$. Finally, let $\underline{s}\mapsto T^i_{(j,l), \cdot} \underline{s}$ denote the departure of a route $i$ packet from position $l$ of queue $j$ which are then acknowledged at congestion window $i$.

We define our queueing system corresponding to congestion level $c$ to be a continuous time Markov chain with the following transition rates
\begin{equation}\label{rates} 
q(\underline{s},\underline{s}')=
\begin{cases}
g^{(c)}_i(\bar{m}_i)\delta_j(l,m_j+1) & \text{for } \underline{s}'=T^i_{\cdot, (j,l)} \underline{s},\;\; j=j_1^i,\\
& \;\; l=1,...,m_j+1,\\
C_j\gamma_j(l,m_j)\delta_{j'}(l',m_{j'}+1) & \text{for } \underline{s}'=T^i_{(j,l),(j',l')}\underline{s},\; j=j_k^i,\; \\
&\;j'=j_{k+1}^i\; k=1,...,k_i-1,\\
&\;i^j_l=i\; l'=1,...,m_{j'}+1,\\
C_j\gamma_j(l,m_j)&\text{for } \underline{s}'=T^i_{(j,l),\cdot} \underline{s},\; j=j_{k_i}^i,\; i_l^j=i.\\
0 &\text{otherwise.}
\end{cases}
\end{equation}
Also like in Section \ref{multi-class queue}, we let $m=(m_{ji}:(j,i)\in \mK)$ record the number of route $i$ packets at queue $j$. Also recall the expressions for $m_j$ and $\bar{m}_i$, (\ref{mj}) and (\ref{mbar}), the number of packets in each queue and in each congestion window. Finally recall we define $G^{(c)}_i$ from $g^{(c)}_i$ by the relation $g^{(c)}_i(\bm_i)=e^{G^{(c)}_i(\bm_i+1)-G^{(c)}_i(\bm_i)}$.

\subsection{Quasi-reversibility and stationary behaviour}
From Proposition \ref{quasi cwnd} and Proposition \ref{quasi queue}, our queueing system consists of a network of quasi-reversible nodes. Thus as considered in Kelly \cite{Ke81} networks of quasi-reversible nodes have a stationary distribution that is described by multiplying the distributions (\ref{cwnd ed}) and (\ref{queue ed}). We prove this in the following theorem.

\begin{theorem}\label{quasi system}
For a stationary queueing system defined by rates (\ref{rates}). Let $M=(M_{ji}:(j,i)\in\mK)$ record the stationary number of packets of each route at each queue, then $M$ has distribution,
\begin{equation}
 \bP(M=m)=\frac{1}{B_{G^{(c)}}} \prod_{j\in\mI} \left( \begin{array}{cc} m_j \\ m_{ji}\: :\: i\ni j \end{array} \right) \frac{1}{C_j^{m_j}} \times \prod_{i\in\mI} e^{G^{(c)}_i(\bar{m}_i)},\qquad m\in\bZ_+^K,\label{ed}
\end{equation}
where,
\begin{equation*}
 B_{G^{(c)}}=\sum_{m\in\bZ_+^K} \prod_{j\in\mJ} \left( \begin{array}{cc} m_j \\ m_{ji}\: :\: i\ni j \end{array} \right)\frac{1}{C_j^{m_j}} \times \prod_{i\in\mI} e^{G^{(c)}_i(\bar{m}_i)}.
\end{equation*}
\end{theorem}
Note that distribution $\bP(M=m)$ is not a product form stationary distribution because we require the constraint $\bar{m}_i=\sum_{j\in i} m_{ji}$, $\forall i\in\mI$.
\begin{proof}
A good candidate for the time reversal of this queueing system is defined by rates $\tilde{q}(\cdot,\cdot)$, where packets follow route $i$ in reverse order $(j_{k_i}^i,...,j^i_1)$, where queues $j\in\mJ$ operate at capacity $C_j$ with $\tilde{\delta}(l,m_j)=\gamma_j(l,m_j)$ and $\tilde{\gamma}_j(l,m_j)=\delta_j(l,m_j)$ and where, as before, congestion window $i\in\mI$ sends packets into the network at rate $g^{(c)}_i(\bar{m}_i)=e^{G^{(c)}_i(\bar{m}_i+1)-G^{(c)}_i(\bar{m}_i)}$. We show
\begin{equation*}
 \pi^{(c)}(\underline{s})=\prod_{j\in\mJ} \frac{1}{C_j^{m_j}}\times \prod_{i\in\mI} e^{G^{(c)}_i(\bar{m}_i)},
\end{equation*}
forms an invariant measure for our explicit Markov chain description. We verify Kelly's Lemma \cite[Theorem 1.13]{Ke79} for our three types of transition: packet injections from congestion windows, transitions between queues, and acknowledgements at congestion windows.

For a packet injected by a congestion window, for $j=j_1^i, \;\; l=1,...,m_j+1,$
\begin{equation*}
\frac{ q(\underline{s},T_{\cdot,(j,l)}^i\underline{s})}{\tilde{q}(T^i_{\cdot, (j,l)} \underline{s}, \underline{s})} = \frac{g^{(c)}_i(\bar{m}_i)\delta_j(l,m_j+1)}{C_j {\gamma}_j(l,m_j+1)}=\frac{\pi^{(c)}(T^i_{\cdot, (j,l)}\underline{s})}{\pi^{(c)}(\underline{s})}.
\end{equation*}
For a transition between queues: for $j=j_k^i$, $j'=j_{k+1}^i,$ $k=1,...,k_i-1$, $i^j_l=i$, $l'=1,...,m_{j'}+1$,
\begin{equation*}
\frac{ q(\underline{s},T_{(j,l)(j',l')}^i\underline{s})}{\tilde{q}(T^i_{(j,l), (j',l')} \underline{s}, \underline{s})}=\frac{C_j\gamma_j(l,m_j)\delta_{j'}(l',m_j'+1)}{C_{j'}\tilde{\gamma}_{j'}(l,m_{j'}+1)\tilde{\delta}_j(l,m_j)}=\frac{\pi^{(c)}(T^i_{(j,l),(j',l'))}\underline{s})}{\pi^{(c)}(\underline{s})}.
\end{equation*}
For an acknowledgement at a congestion window $i$, for $j=j_{k_i}^i$, $i=i_l^j$ 
\begin{equation*}
 \frac{q(\underline{s},T^i_{(j,l),\cdot}\underline{s})}{\tilde{q}(T^i_{(j,l),\cdot} \underline{s},\underline{s})}=\frac{C_j\gamma_j(l,m_j)}{g^{(c)}_i(\bar{m}_i-1)\tilde{\delta}_j(l,m_j)}=\frac{\pi^{(c)}(T^i_{(j,l),\cdot}\underline{s})}{\pi^{(c)}(\underline{s})}.
\end{equation*}
We note that the transition intensity of the reversed process agrees with that of the forward process
\begin{align*}
 \tilde{q}(\underline{s})&=-\sum_{\underline{s}'} \tilde{q}(\underline{s},\underline{s}')=\sum_{i\in\mI} g^{(c)}_i(\bar{m}_i) + \sum_{j: m_j>0} C_j =q(\underline{s}).
\end{align*}
This verifies Kelly's Lemma and thus $\pi^{(c)}(\underline{s})$ is an invariant measure. Note $B_{G^{(c)}}$ is expressable as 
\begin{equation*}
 B_{G^{(c)}}=\sum_{m\in\bZ_+^K} \prod_{j\in\mJ}\left[ \left( \begin{array}{cc} m_j \\ m_{ji}\: :\: i\ni j \end{array} \right)\prod_{i: j\in i}\left( \frac{e^{\lambda_i}}{C_j}\right)^{m_{ji}} \right]\times \prod_{i\in\mI} e^{G^{(c)}_i(\bar{m}_i)-\lambda_i \bar{m}_i},
\end{equation*}
for all $\lambda\in\bR^I$. By Assumption 2 the function $\bar{m}_i\mapsto G^{(c)}_i(\bar{m}_i)-\lambda_i \bar{m}_i$ is bounded from above for all $\lambda_i$. Applying this upper bound and choosing $\lambda\in\bR^I$ such that $\sum_{j\in i} e^{\lambda_i} < C_j$, the sum for $B_{G^{(c)}}$ is finite. Thus summing over states of invariant measure $\pi^{(c)}(\underline{s})$ gives the stationary distribution (\ref{ed}).
\begin{flushright}
$\square$ 
\end{flushright}
\end{proof}

\subsection{Large deviations}\label{LDP system}

We now study the large deviations behaviour of our stationary queueing systems as $c\rightarrow\infty$. As $c$ increases the congestion windows increase the number of packets within the queueing system and we study its large deviations behaviour. Once again we think of each congestion controller attempting to exploit network capacity by increasing $c$ and congesting the network. We will relate the most likely state in our queueing system to the solution of the system problem.

We use the same notation from Sections \ref{cwnd} and \ref{multi-class queue}. We, also, assume Assumptions \ref{assump 1} and \ref{assump 2} hold. As in Section \ref{cwnd}, we consider a sequence of congestion windows defined by $G_i^{(c)}(k)=cG_i(\frac{k}{c}+\frac{d_i^{(c)}}{c})$ for $i\in\mI$, $c\in\bN$. Here $G_i$ is expressible in terms of utility function $U_i$ by (\ref{Ui}). We define the function,
\begin{equation}
 \beta_{G,\lambda}(m,\tilde{m})=\sum_{\substack{(j,i)\in\mK:\\m_{ji}>0}} m_{ji}\log \frac{m_{ji}C_j}{m_je^{\lambda_i}} -\sum_{i\in\mI} \{G_i(\tilde{m}_i)-\lambda_i\bar{m}_i\},\label{lagrangian}
\end{equation}
for $m\in\bR_+^K$, $\tilde{m}\in\bR_+^I$ and $\lambda_i\in\bR^I$. We use the shorthand that $\beta_G=\beta_{G,\lambda}$ for $\lambda=0$ and we also use the shorthand $\beta_G(m)=\beta_G(m,\tilde{m})$ when $\tilde{m}_i=\sum_{j\in i} m_{ji}$ for all $i\in\mI$. We define,
\begin{align}
 \beta_G^*=&\min \sum_{\substack{(j,i)\in\mK:\\m_{ji}>0}} m_{ji}\log \frac{m_{ji}C_j}{m_j} -\sum_{i\in\mI} G_i(\tilde{m}_i)\quad\text{subject to}\quad \sum_j m_{ji} =\bar{m}_i,\quad i\in\mI\label{primal1}\\
&\text{over}\quad  m\in\bR_+^K,\quad \bar{m}\in\bR_+^I\label{primal3}. 
\end{align}
For each $c$ fixed, consider a stationary queueing system defined by rates (\ref{rates}) with congestion windows defined by $G_i^{(c)}(\cdot)$, $i\in\mI$. Let $M^{(c)}=(M_{ji}^{(c)}: (j,i)\in\mK)$ record the stationary distribution of the number of packets on each route at each queue in our queueing system, (\ref{ed}). We characterise the large deviations of our sequence of queueing systems with the following theorem.

\begin{theorem}\label{LDP}
The sequence $\frac{M}{c}^{(c)}$, $c\in\bN$ obeys a large deviations principle with good rate function $\beta_G(\cdot)-\beta_G^*$. That is for all $D\subset \bR_+^K$,
\begin{align*}
 -\!\!\inf_{m\in D^{\circ}} \!(\beta_G(m)-\beta_G^*) &\leq \liminf_{c\rightarrow \infty} \bP^{(c)}\Big(\frac{M}{c}^{(c)}\!\!\!\in D\Big)\\
& \leq \limsup_{c\rightarrow \infty} \bP^{(c)}\Big(\frac{M}{c}^{(c)}\!\!\!\in D\Big) \leq  -\!\!\inf_{m\in \bar{D}}\! (\beta_G(m)-\beta_G^*) .
\end{align*}
\end{theorem}
\begin{proof}
Assuming $\lambda\in\bR^I$ satisfies
\begin{equation*}
 \sum_{i: j\in i} e^{\lambda_i} < C_j, \qquad \forall j\in\mJ.
\end{equation*} 
We define a product form stationary distribution on $\bZ_+^K\times\bZ_+^I$ with,
\begin{align}\label{P Ptilde}
 \tilde{\bP}^{(c)} (M^{(c)}\!\!=m&,\tilde{M}^{(c)}\!\!=\tilde{m})=\frac{1}{\tilde{B}^{(c)}} \prod_{j\in\mJ}\! \left(\!\! \begin{array}{cc} m_j \\ m_{ji} \: : \: i\ni j\end{array}\!\! \right) \prod_{i:\: j\in i} \frac{e^{\lambda_im_{ji}}}{C_j^{m_{ji}}} \times \prod_{i\in\mI} e^{G_i^{(c)}(\bar{m}_i)-\lambda_i \bar{m}_i},
\end{align}
$m\in\bZ_+^K$ and $\tilde{m}\in\bZ_+^I$ where,
\begin{equation*}
 \tilde{B}^{(c)}= \prod_{j\in\mJ} \left( \frac{C_j}{C_j - \sum_{i : j\in i} e^{\lambda_i} } \right) \times \prod_{i\in\mI} \left( \sum_{k=0}^\infty e^{G_i^{(c)}(k) -\lambda_i k} \right).
\end{equation*}
Note by (\ref{finite sum}) in Proposition \ref{LDP cwnd prop}, $\tilde{B}^{(c)}$ is finite. Note that $\bP^{(c)}$ is expressible in terms of $\tilde{\bP}^{(c)}$ through the conditional probability,
\begin{equation*}
 \bP^{(c)}(M^{(c)}=m)=\tilde{P}^{(c)}( M^{(c)}=m\: |\: \tilde{M}^{(c)}_i=\sum_{j\in i} M^{(c)}_{ji},\; i\in\mI ),
\end{equation*}
for $m\in\bZ_+^K$. By proving large deviations results about $\tilde{\bP}^{(c)}$ we are able to prove a large deviations principle for $\bP^{(c)}$. First, Proposition \ref{LDP queue prop}, Proposition \ref{LDP cwnd prop}a) and the definition of $U_i$ (\ref{Ui}), we have that
\begin{equation*}
 \lim_{c\rightarrow\infty} \frac{1}{c} \log \tilde{B}^{(c)} = -\sum_{i\in\mI} U_i(e^{\lambda_i}).
\end{equation*}
Thus by Proposition \ref{LDP cwnd prop} and Proposition \ref{LDP queue prop}, for $m\in\bR_+^K$, $\tilde{m} \in\bR_+^I$ with bounded sequences $\sigma^{(c)}\in\bR_+^K$ and $\tilde{\sigma}^{(c)}\in\bR_+^I$ $c\in\bN$ such that $cm+\sigma^{(c)}\in\bZ_+^K$ and $c\tilde{m} + \tilde{\sigma}^{(c)} \in\bZ_+^I$ we have that
\begin{equation*}
 -\beta_{G,\lambda}(m,\tilde{m})+\sum_{i\in\mI} U_i(e^{\lambda_i}) =\lim_{c\rightarrow \infty} \frac{1}{c} \log \tilde{\bP}^{(c)}(M^{(c)}=cm+\sigma^{(c)},\: \tilde{M}^{(c)}= c\tilde{m} + \tilde{\sigma}^{(c)}).
\end{equation*}
Take $E\subset \bR_+^K\times \bR^I_+$, either open or more generally such that $\forall (m,\tilde{m})\in E$ there exists a sequence as described above with $cm+\sigma^{(c)}\in\bZ_+^K$ and $c\tilde{m} + \tilde{\sigma}^{(c)} \in\bZ_+^I$ such that $(m+\frac{\sigma^{(c)}}{c}, \tilde{m}+\frac{\tilde{\sigma}^{(c)}}{c})\in E$ eventually as $c\rightarrow\infty$ then 
\begin{align}
 &\;\;\;\;-\inf_{(m,\tilde{m})\in E} \beta_{G,\lambda}(m,\tilde{m})\notag\\ 
&= -\inf_{(m,\tilde{m})\in E} \lim_{c\rightarrow \infty} \frac{1}{c} \log \tilde{\bP}^{(c)} (M^{(c)}=cm+\sigma^{(c)},\: \tilde{M}^{(c)}=c\tilde{m}+\tilde{\sigma}^{(c)})\notag \\
&\leq \liminf_{c\rightarrow \infty} \frac{1}{c} \log \tilde{\bP}^{(c)} \Big( ( \frac{M}{c}^{(c)}, \frac{\tilde{M}}{c}^{(c)}) \in E \Big).\label{LDP low}
\end{align}
This gives us a large deviations lower bound for $\tilde{\bP}^{(c)}$. We prove the upperbound by using the G\"artner-Ellis Theorem \cite[Page 44]{DeZe98}. We study the moment generating function of $\tilde{\bP}^{(c)}$, for $\theta\in\bR^K$ and $\phi\in\bR^I$
\begin{align*}
&\bE e^{\theta\cdot M^{(c)}+ \phi\cdot \tilde{M}^{(c)}} \\
&= 
\begin{cases}
 \prod_{j\in\mJ} \left( \frac{C_j-\sum_{i: j\in i} e^{\lambda_i}}{C_j-\sum_{i: j\in i} e^{\lambda_i+ \theta_{ji}}}\right)\\ 
\qquad\times \prod_{i\in\mI} \left( \frac{\sum_{k=0}^\infty e^{G_i^{(c)}(k)-(\lambda_i-\phi_i)k }}{ \sum_{k=0}^{\infty} e^{G_i^{(c)}(k) - \lambda_ik}}\right)&\text{if } \sum_{i:j\in i} e^{\lambda_i+\theta_{ji}} < C_j,\quad j\in\mJ, \\
\infty & \text{otherwise.}
\end{cases}
\end{align*}
Thus combining Proposition \ref{LDP cwnd prop} part a) and Proposition \ref{LDP queue prop} part b),
\begin{align*}
 F(\theta,\phi)&=\lim_{c\rightarrow\infty} \frac{1}{c} \log \bE e^{\theta\cdot\frac{M}{c}^{(c)}+ \phi\cdot \frac{\tilde{M}}{c}^{(c)}}\\
&=\begin{cases}
   \sum_{i\in\mI} U_i(e^{\lambda_i})-U_i(e^{\lambda_i-\phi_i}) & \text{if } \sum_{i: j\in i} e^{\lambda_i+\theta_{ji}} < C_j,\;\; j\in\mJ\\
\infty & \text{otherwise.} 
  \end{cases}
\end{align*}
Thus $F$ has Legendre-Fenchel transform,
\begin{align*}
F^*(m,\bar{m})=&\max_{\substack{\theta\in\bR^K\\ \phi\in\bR^I}} \sum_{ji} m_{ji}\theta_{ji} + \sum_i \tilde{m}_i\phi_i + \sum_i\Big( U_i(e^{\lambda_i-\phi_i}) - U_i(e^{\lambda_i})\Big)\\
&\text{subject to}\:\: \sum_{i: j\in i} e^{\lambda_i+\theta_{ji}}<C_j, \;\; j\in\mJ\\
=&\max_{\theta\in\bR^K} \left\{ \sum_{ji} m_{ji}\theta_{ji}\: :\: \sum_i e^{\lambda_i+\theta_{ji}}<C_j,\:\: j\in\mJ\right\}\\
+& \sum_{i\in\mI} \max_{\phi'_i\in\bR} \left\{ U_i(e^{\lambda_i+\phi'_i})-\tilde{m}_i(\phi'_i+\lambda_i)\right\}\\
+&\sum_i \lambda_i\tilde{m}_i - \sum_i U_i(e^{\lambda_i})\\
=& \sum_{ji} m_{ji}\log \frac{m_{ji}C_j}{m_j e^{\lambda_i}} - \sum_i ( G_i(\tilde{m}_i)-\lambda_i\tilde{m}_i) - \sum_i U_i(e^{\lambda_i}).
\end{align*}
In the second equality we collected terms and substitute $\phi'_i=-\phi_i$ for $i\in\mI$.
In the third equality we apply (\ref{beta LF}) to the first maximization and the user problem (\ref{Gi}) to the second maximization. From this the G\"artner-Ellis Theorem \cite[Page 44]{DeZe98} proves that for all closed sets $E\subset \bR_+^K\times \bR_+^I$
\begin{equation}\label{LDP up}
 \limsup_{c\rightarrow\infty} \frac{1}{c}\log \tilde{\bP}^{(c)}\Big( (\frac{M}{c}^{(c)}, \frac{\tilde{M}}{c}^{(c)})\in E \Big)\leq -\!\!\inf_{(m,\tilde{m})\in E} \Big( \beta_{G,\lambda} (m,\tilde{m}) - \sum_i U_i(e^{\lambda_i})\Big).
\end{equation}
In particular we are interested in the closed set $\bar{E}=\{ (m,\tilde{m}): \sum_{j\in i} m_{ji}=m_i,\: i\in\mI\}$. Note if $m\in\bR_+^I$ and sequence $\sigma^{(c)}\in\bR_+^I$, $c\in\bN$, is such that $cm+\sigma^{(c)}\in\bZ_+^I$ and defining $\bar{\sigma}^{(c)}_i=\sum_{j\in i} \sigma^{(c)}_{ji}$ then $(cm+\sigma^{(c)},c\bar{m}+\bar{\sigma}^{(c)})\in \bar{E}$. So we may apply lower bound (\ref{LDP low}) to $\bar{E}$ and also apply upper bound (\ref{LDP up}) to this choice of $\bar{E}$. Hence we have that,
\begin{equation}
 \lim_{c\rightarrow\infty} \frac{1}{c} \log \tilde{\bP}^{(c)}(\sum_{j\in i} M_{ji}^{(c)} = \tilde{M}_i^{(c)},\: i\in\mI)=-\beta^*_G-\sum_i U_i(e^{\lambda_i}). \label{LDP scaling 1}
\end{equation}
Or put otherwise we have for the normalising constant $B_{G^{(c)}}$,
\begin{equation}
 \lim_{c\rightarrow\infty} \frac{1}{c} \log B_{G^{(c)}} = -\beta^*_G. \label{LDP scaling 2}
\end{equation}
From (\ref{P Ptilde}) combining (\ref{LDP scaling 1}) with large deviations upper bound (\ref{LDP up}), for all closed sets $D\in\bR_+^K$ letting $D'=\{ (m,\bar{m}): m\in D\}\subset \bar{E}$ we have
\begin{align*}
 &\limsup_{c\rightarrow \infty} \frac{1}{c}\log \bP^{(c)}\Big(\frac{M}{c}^{(c)}\in D\Big)\\
=& \limsup_{c\rightarrow\infty} \bigg(\frac{1}{c}\log \tilde{\bP}^{(c)} \Big((\frac{M}{c}^{(c)},\frac{\tilde{M}}{c}^{(c)})\in D'\Big) - \frac{1}{c}\log \tilde{\bP}( \sum_j M_{ji}^{(c)} = \tilde{M}_i^{(c)},\: i\in\mI)\bigg)\\
\leq &  -\inf_{m\in D} \beta_G(m)\;\;  +\;\;  \beta^*_G.
\end{align*}
This proves the large deviations upperbound for $\bP^{(c)}$. The lower bound follows similarly by combining (\ref{LDP scaling 1}) with lower bound (\ref{LDP low}).
\begin{flushright}
$\square$ 
\end{flushright}
\end{proof}

\subsection{Duality between state and flow}
An important phenomenon we find from our large deviation analysis is that the limiting state of our queueing system and the limiting flow through our queueing system are dual. We demonstrate here that dual form of the optimization problem (\ref{primal1}-\ref{primal3}) found in our large deviations analysis is exactly the system problem (\ref{system1}-\ref{system3}). 

\begin{theorem}\label{primaldual theorem}
\begin{align}
 \beta_G^*=&\min_{\substack{m\in\bR_+^K\\\bar{m}\in\bR_+^I}} \sum_{\substack{(j,i)\in\mK:\\m_j>0}} m_{ji}\log \frac{m_{ji}C_j}{m_j} - \sum_{i\in\mI} G_i(\bar{m}_i) \;\text{subject to } \sum_{j\in i} m_{ji}=\bar{m}_i,\;\: i\in\mI\label{primal}\\
=&\max_{\Lambda\in\bR_+^I}\quad \sum_{i\in\mI} U_i(\Lambda_i) \quad\text{ subject to }\quad\sum_{i:j\in i} \Lambda_i \leq C_j,\quad j\in\mJ\label{dual}.
\end{align}
Moreover, vector $(m^*,\bar{m}^*)\in\bR_+^K\times\bR_+^I$ optimizes (\ref{primal}) and $\Lambda^*\in\bR_+^I$ optimizes (\ref{dual}) iff
\begin{align}
&\sum_{j\in i} m^*_{ji} =\bar{m}_i^*,\quad \forall i\in\mI,\\ 
&\sum_{i: j\in i} \Lambda_i^* \leq C_j,\quad \forall j\in\mJ\\
& m^*_j\Lambda^*_i = m_{ji}^*C_j, \quad \forall (j,i)\in\mK,\label{equiv2}\\
&e^{G_i'(\bar{m}_i^*)}=\Lambda_i^*,\quad \forall i\in\mI. \label{equiv}
\end{align}
Here $G'_i$ is the derivative of the function $G_i$.
\end{theorem}
\begin{proof}
Note the Lagrangian of optimisation problem (\ref{primal}) with Lagrange multipliers $\lambda=(\lambda_i:i\in\mI)$ is exactly $\beta_{G,\lambda}$ as defined by (\ref{lagrangian}).
Minimizing this Lagrangian we have
\begin{align}
 \min_{\substack{m\in\bR_+^K\\ \bar{m}\in\bR_+^I}}& \beta_{G,\lambda}(m,\bar{m})= \min_{m,\bar{m}} \sum_{ji} m_{ji}\log\frac{m_{ji}C_j}{m_j} - \sum_{i\in\mI} G_i(\bar{m}) + \sum_i \lambda_i\Big(\bar{m}_i-\sum_{j\in i} m_{ji}\Big)\notag\\
&=\min_{m\in\bR_+^K}\Big\{ \sum_{ji} m_{ji}\log \frac{m_{ji}C_j}{m_je^{\lambda_i}} \Big\} - \sum_i \max_{\bar{m}_i>0} \Big\{ G_i(\bar{m}_i) -  \lambda_i \bar{m}_i \Big\}\label{lagrange min}\\
&=
\begin{cases} 
 \sum_i U_i(e^{\lambda_i}) & \text{if } \sum_{i: j\in i} e^{\lambda_i} \leq C_j,\\
-\infty & \text{otherwise.}
\end{cases}\notag
\end{align}
In the final, equality we apply Proposition \ref{LDP queue prop} b) and definition (\ref{Ui}).
Thus the dual of optimization problem (\ref{primal}) is,
\begin{equation*}
 \max_{\lambda\in\bR^I}\quad \sum_{i\in\mI} U_i(e^{\lambda_i})\quad \text{ subject to }\quad\sum_{i:j\in i} e^{\lambda_i} \leq C_j,\quad j\in\mJ.\label{dual2}
\end{equation*}
By the strong duality of optimization problem (\ref{primal}) and (\ref{dual}) we have that expressions (\ref{primal}) and (\ref{dual}) are equal.

Now we demonstrate the only if part of the equivalence with (\ref{equiv}). Suppose $\Lambda^*$ optimises (\ref{dual}). We now consider how our Lagrangian behaves for $\lambda=(\log \Lambda^*_i:\: i\in\mI)$. From (\ref{lagrange min}) we see $(m^*,\bar{m}^*)$ must solve
\begin{align}
\min_{\bar{m}_i>0} \bigg\{ G_i(\bar{m}_i) -\bar{m}_i \log \Lambda^*_i \bigg\}&= -U_i(\lambda^*_i),\quad i \in\mI,\label{G min}\\
 \min_{\substack{m\in\bR:\\ m_j>0}} \bigg\{ m_j \sum_{i:j\in i} \frac{m_{ji}}{m_j} \log \frac{m_{ji}C_j}{m_j \Lambda^*_i}  \bigg\}&=0, \quad j\in\mJ.\label{m min}
\end{align}
From (\ref{G min}) we see that $e^{G'_i(\bar{m}^*_i)} = \Lambda_i^*$. We now consider (\ref{m min}) if $m^*_j=0$ then (\ref{m min}) is satisfied and so is (\ref{equiv2}). If $m^*_j>0$ then given relative entropy result (\ref{entropy}), (\ref{m min}) can only hold if $\sum_{i: j\in i} \Lambda_i = C_j$ and $\frac{m^*_{ji}}{m^*_j}=\frac{\Lambda_i}{C_j}$ for all $i\in\mI$ such that $j\in i$. Thus both (\ref{equiv2}) and (\ref{equiv}) hold.

Conversely if (\ref{equiv2}) and (\ref{equiv}) hold then substituting $m^*$ and $\bar{m}^*$ into the objective function of optimization problem (\ref{primal}) gives that
\begin{align*}
 \sum_{\substack{(j,i)\in\mK\\m^*_j>0}} m^*_{ji}\log \frac{m^*_{ji}C_j}{m^*_j} - \sum_{i\in\mI} G_i(\bar{m}^*_i)&=\sum_{i\in\mI} \bar{m}^*_i \log \Lambda_i^* - \sum_{i\in\mI} G_i(\bar{m}^*_i)= \sum_{i\in\mI} U_i(\Lambda_i^*).
\end{align*}
In the final equality, we use Legendre-Fenchel transform expression (\ref{Ui}). Thus $\bar{m}_i^*$ attains the value $\sum_i U_i(\Lambda_i^*)$ for $\Lambda^*$ feasible and thus by the Lagrangian Sufficiency Theorem our solutions are optimal for (\ref{primal}) and (\ref{dual}).
\begin{flushright}
$\square$ 
\end{flushright}
\end{proof}

\subsection{Most likely state and convergence of throughput}
We now study the most likely behaviour of our stationary sequence of queueing systems $\frac{M}{c}^{(c)}$, $c\in\bN$. Section \ref{LDP system} suggests our sequence of queueing systems implicitly solve the system problem (\ref{system1}-\ref{system3}). This section formalizes this assertion. We show the sequence of stationary queueing systems $\frac{M}{c}^{(c)}$, $c\in\bN$ concentrates on the solutions of the primal optimization problem (\ref{primal1}-\ref{primal3}). From this we show the stationary rate packets pass through the network converges to the rate allocation solving the system problem. We define the manifold
\begin{equation*}
 \mM= \{ m\in\bR_+^K: m_{ji}C_j=m_j\Lambda_i^*,\;(j,i)\in\mK,\; G_i'(\bar{m}_i)=\log \Lambda_i^*,\; i\in\mI\},
\end{equation*}
where $(\Lambda_i^*: i\in\mI)$ is the optimal solution to the system problem (\ref{system1}-\ref{system3}). From Theorem \ref{primaldual theorem} we know that $\mM$ is the set of solutions to the primal optimization problem (\ref{primal1}-\ref{primal3}). The stationary sequence of queues $\frac{M}{c}^{(c)}$, $c\in\bN$, considered in Section \ref{LDP system} converges in probability to the set of solutions $\mM$.
\begin{theorem}
\begin{equation*}
 \bP^{(c)}\Big( \inf_{m\in\mM} \Big|\Big| \frac{M}{c}^{(c)} - m \Big|\Big| \geq \epsilon \Big) \xrightarrow[c\rightarrow\infty]{} 0.
\end{equation*}
\end{theorem}
\begin{proof}
For $\epsilon>0$, let $\mM_{\epsilon}=\{ m\in\bR_+^K: \inf_{m'\in\mM}||m-m'||<\epsilon \}$. As $\mM$ is closed and compact
\begin{align*}
 \beta_{G,\epsilon}^*:=& \min_{\substack{m\in\bR_+^K\\m\in\bR_+^I}} \sum_{\substack{(j,i)\in\mK:\\ m_j>0}} m_{ji} \log \frac{m_{ji}C_j}{m_j} -\sum_i G_i(\bar{m}_i)\\
&\text{subject to}\;\; m\notin \mM_{\epsilon},\; \sum_{j\in i} m_{ji} =\bar{m}_i,\quad i\in\mI\\
>& \beta^*_{G},
\end{align*}
where we recall $\beta^*_G$ from (\ref{primal1}-\ref{primal3}). Thus by Theorem \ref{LDP} for all $\epsilon>0$
\begin{equation*}
 \limsup_{c\rightarrow\infty} \frac{1}{c} \log \bP^{(c)} \Big( \inf_{m\in\mM} \big|\big| \frac{M}{c}^{(c)}-m\big|\big| \geq \epsilon \Big) \leq -\beta_{G,\epsilon}^* + \beta_G^*. 
\end{equation*}
Thus $\forall \epsilon'\in (0, \beta_{G,\epsilon}^*-\beta_G^*)$ eventually as $c\rightarrow\infty$
\begin{equation*}
 \bP^{(c)}\Big( \inf_{m\in\mM} \big|\big|\frac{M}{c}^{(c)}-m\big|\big| \geq \epsilon \Big) \leq e^{-c(\beta_{G,\epsilon}^*-\beta_G^*) + c\epsilon'}\xrightarrow[c\rightarrow\infty]{} 0.
\end{equation*}
\begin{flushright}
$\square$ 
\end{flushright}
\end{proof}

The stationary throughput of route $i$ packets in our queueing system can be expressed as
\begin{equation*}
 \Lambda_i^{(c)} =\bE^{(c)} e^{G_i^{(c)}(\bar{M}^{(c)}_i+1)-G_i^{(c)}(\bar{M}_i)},\qquad i\in\mI.
\end{equation*}
That is the stationary rate packets are sent into the queueing network by the $i$-th congestion window when it is at congestion level $c$. We now show that this rate converges to the solution to the system problem $(\Lambda_i^*:i\in\mI)$. In this sense our sequence of queueing systems implicitly solve the system problem.
\begin{theorem}\label{final thrm}
\begin{equation*}
 \Lambda_i^{(c)}\xrightarrow[c\rightarrow\infty]{} \Lambda^*_i,\qquad\quad i\in\mI.
\end{equation*}
\end{theorem}
\begin{proof}
 We first describe a modification of measure $\bP^{(c)}$ that will be useful to us. For fixed $i\in\mI$ we let $\breve{\bP}^{(c)}$ be the stationary distribution of a queueing system defined by the same rates as $\bP^{(c)}$ except that the $i$-th congestion window is defined by $\breve{G}^{(c)}_i(\bar{m}_i)=G_i^{(c)}(\bar{m}_i+1)$. Observe for all $m\in\bZ_+^K$
\begin{align}
 &\bE^{(c)}\Big[ e^{G_i^{(c)}(\bar{M}^{(c)}_i+1)-G_i^{(c)}(\bar{M}^{(c)}_i)} \bI[M^{(c)}=m]\Big]\\
=&\frac{1}{B_G} \prod_{j\in\mI} \left( \begin{array}{cc} m_j \\ m_{jr}\: :\: r\ni j \end{array} \right) \frac{1}{C_j^{m_j}} \times \prod_{r\in\mI} e^{\breve{G}_r(\bar{m}_r)}=\frac{B_{\breve{G}^{(c)}}}{B_{G^{(c)}}}\breve{\bP}^{(c)}(M^{(c)}=m). \label{subst}
\end{align}
\noindent Also by definition $G_i^{(c)}(\bar{m}_i)=G_i\big(\frac{\bar{m}_i}{c}+\frac{d_i^{(c)}}{c}\big)$, thus $\breve{G}_i^{(c)}$ corresponds to taking $\breve{d}_i^{(c)}=d_i^{(c)}+1$. The precise values of the bounded sequence $\{d_i^{(c)}\}_{c\in\bN}$ do not determine any of the large deviations behaviour of $\bP^{(c)}$, thus both $\bP^{(c)}$ and $\breve{\bP}^{(c)}$ exhibit exactly the same large deviations behaviour. Hence given (\ref{LDP scaling 2}) and also Theorem \ref{LDP} we have that
\begin{equation}
 \lim_{c\rightarrow\infty} \frac{1}{c} \log \frac{B_{\breve{G}^{(c)}}}{B_{G^{(c)}}} = 0,\qquad \lim_{c\rightarrow\infty}\breve{\bP}^{(c)}\Big( \inf_{m\in\mM} || \frac{M}{c}^{(c)} - m || \geq \epsilon \Big)=0. \label{2lims}
\end{equation}

Now let $(m^*,\bar{m}^*)$ be an optimal solution to the primal optimization problem (\ref{primal1}-\ref{primal3}). By assumption  $G_i$ is differentiable with a continuous derivative at $\bar{m}^*_i$. Thus by the Mean Value Theorem $\forall \epsilon'>0$ $\exists \epsilon>0$ and $c'\in\bN$ such that $\forall\: c>c'$ and $\forall \bar{m}_i>0$ with $|\bar{m}_i - \bar{m}^*_i|\leq J\epsilon$, we have that
\begin{align*}
&|e^{G_i^{(c)}(\bar{m}_i+1)-G_i^{(c)}(\bar{m}_i)} - e^{G_i'(\bar{m}^*_i)}|\\
=&\Big|\exp\Big\{\frac{G_i(\bar{m}_i+\frac{1}{c}+\frac{d^{(c)}_i}{c})-G_i(\bar{m}_i+\frac{d^{(c)}_i}{c})}{\frac{1}{c}}\Big\}-\exp\{G_i'(\bar{m}^*_i)\}\Big|\leq \epsilon'.
\end{align*}
We can now show that eventually as $c\rightarrow\infty$,
\begin{align*}
&|\Lambda^{(c)}_i-\Lambda^*_i|\\
 \leq &\bE \Big| e^{G_i^{(c)}(\bar{M}_i^{(c)}+1)-G_i^{(c)}(\bar{M}_i^{(c)})}-e^{G_i'(\bar{m}^*_i)}\Big|\\
\leq &\;\epsilon' \bP^{(c)}\Big( \big|\frac{\bar{M}_i^{(c)}}{c} -\bar{m}^*_i \big|< J\epsilon \Big) + e^{G'_i(\bar{m}^*_i)}\bP^{(c)}\Big( \big|\frac{\bar{M}_i^{(c)}}{c} -\bar{m}^*_i \big|\geq J\epsilon \Big)\\
& +\;\bE^{(c)} e^{G_i^{(c)}(\bar{M}^{(c)}_i+1)-G_i^{(c)}(\bar{M}^{(c)}_i)} \bI\Big[ \big|\frac{\bar{M}_i^{(c)}}{c} -\bar{m}^*_i \big|\geq J\epsilon \Big]\\
\leq & \epsilon' +e^{G'_i(\bar{m}^*_i)}\bP^{(c)}\Big( \inf_{m\in\mM} || \frac{M}{c}^{(c)} - m || \geq \epsilon\Big)\\
& + \frac{B_{\breve{G}^{(c)}}}{B_{G^{(c)}}}\breve{\bP}^{(c)}\Big( \inf_{m\in\mM} || \frac{M}{c}^{(c)} - m || \geq \epsilon\Big)\\
\leq &\; \epsilon' + e^{G'_i(\bar{m}^*_i)} e^{-c(\beta^*_{G,\epsilon}-\beta^*_G) + c\epsilon''} + e^{c\epsilon''}e^{-c(\beta^*_{G,\epsilon}-\beta^*_G) + c\epsilon''}\xrightarrow[c\rightarrow\infty]{} \epsilon'.
\end{align*}
In the second inequality we make the substitution (\ref{subst}),  in the third inequality we apply (\ref{2lims}) and we take $2\epsilon''<\beta^*_{G,\epsilon}-\beta_{G}^*$. As $\epsilon'$ is arbitrary the result holds.

\begin{flushright}
$\square$ 
\end{flushright}
\end{proof}

\section{Conclusion} 
Previous work has considered the solution of the system problem by analysing differential equations. In this paper we have shown that this same notion of utility optimization can be solved by considering queueing networks with end-to-end control. This leads us to consider different interpretations of Kelly's decomposition results where pricing is determined by delay. This work emphasises the duality between the flow through a network and its state and also emphasises a wide variety of fairness that are provably achievable by end-to-end control.

\section*{Acknowledgement}
I would like to thank my PhD supervisor Frank Kelly for his continued support and insight, which has encouraged me to write this paper.

\bibliographystyle{acmtrans-ims}
\bibliography{../../references/references} 
\end{document}